\title{\vspace{-0.1cm} Unavoidable patterns}
\author{
Jacob Fox\thanks{Department of Mathematics, Princeton, Princeton,
NJ 08544. Email: {\tt jacobfox@math.princeton.edu}. Research supported by
an NSF Graduate Research Fellowship and a Princeton Centennial
Fellowship.} \and Benny Sudakov\thanks{Department of Mathematics, UCLA, Los Angeles, 90095. E-mail:
bsudakov@math.ucla.edu.
Research supported in part by NSF CAREER award DMS-0546523, NSF grant
DMS-0355497, and a USA-Israeli BSF grant.}}

\documentclass[11pt]{article}
\usepackage{amsfonts,amssymb,amsmath,latexsym}

\newenvironment{proof}
      {\medskip\noindent{\bf Proof.}\hspace{1mm}}
      {\hfill$\Box$\medskip}

\def\qed{\ifvmode\mbox{ }\else\unskip\fi\hskip 1em plus 10fill$\Box$}

\newtheorem{theorem}{Theorem}[section]

\newtheorem{lemma}[theorem]{Lemma}
\newtheorem{proposition}[theorem]{Proposition}

\newtheorem{conjecture}[theorem]{Conjecture}

\begin{document}
\date{}

\maketitle

\begin{abstract}
Let $\mathcal{F}_k$ denote the family of $2$-edge-colored complete graphs on $2k$ vertices in which one color forms either a clique of order
$k$ or two disjoint cliques of order $k$. Bollob\'as conjectured that for every $\epsilon>0$ and positive
integer $k$ there is an $n(k,\epsilon)$ such that every $2$-edge-coloring of the complete graph of order
$n \geq n(k,\epsilon)$ which has at least $\epsilon {n \choose 2}$ edges in each color contains a member of $\mathcal{F}_k$.
This conjecture was proved by Cutler and Mont\'agh, who showed that $n(k,\epsilon)<4^{k/\epsilon}$.
We give a much simpler proof of this conjecture which in addition shows that $n(k,\epsilon)<\epsilon^{-ck}$ for some constant
$c$. This bound is tight up to the constant factor in the exponent for all $k$ and $\epsilon$. We also discuss similar results for tournaments and hypergraphs.
\end{abstract}

\section{Introduction}

The {\it Ramsey number} $R(k)$ is the least positive integer $n$
such that every $2$-edge-coloring of the complete graph
$K_n$ contains a monochromatic clique of order $k$.  A classical result of
Erd\H{o}s and Szekeres \cite{ErSz}, which is a quantitative version
of Ramsey's theorem, implies that $R(k) < 2^{2k}$ for every
positive integer $k$. Erd\H{o}s \cite{Er3} showed using
probabilistic arguments that $R(k) > 2^{k/2}$ for $k
> 2$. Over the last sixty years, there have been several
improvements on these bounds (see, e.g., \cite{Co}). However,
despite efforts by various researchers, the constant factors in the
above exponents remain the same.

If we are interested in finding a subgraph in a $2$-edge-coloring of $K_n$ that is not monochromatic,
we must assume that each color is sufficiently represented, e.g., that
each color class has at least $\epsilon{n \choose 2}$ edges. Let $\mathcal{F}_k$ denote the family of $2$-edge-colored
complete graphs on $2k$ vertices in which one color forms either a clique of order
$k$ or two disjoint cliques of order $k$. Consider a $2$-edge-coloring of $K_n$ with $n$ even in which one of the colors forms
a clique of order $n/2$ or two disjoint cliques of order $n/2$. Clearly,
these colorings have at least roughly $1/4$ of the edges in each
color, and nevertheless they basically do
not contain any colored patterns except those in $\mathcal{F}_k$. This shows that the
$2$-edge-colorings in $\mathcal{F}_k$ are essentially the only types of
of patterns that are possibly unavoidable in $2$-edge-colorings that are far from being monochromatic. It is also clear
that $\mathcal{F}_k$ forms a minimal such family. Generalizing the classical Ramsey problem, Bollob\'as conjectured that for
every $\epsilon>0$ and $k$ there is an $n_0$ such that every $2$-edge-coloring of
 $K_n$ with $n \geq n_0$ which has at least $\epsilon {n \choose 2}$ edges in each color contains a member of
$\mathcal{F}_k$. This conjecture was proved by
Cutler and Mont\'agh \cite{CuMo}. Let $n(k,\epsilon)$ denote the minimum $n_0$ which satisfies Bollobas's conjecture.
Cutler and Mont\'agh proved that $n(k,\epsilon)<4^{k/\epsilon}$. Here we present a much simpler proof of Bollobas's conjecture
which also gives a better bound on $n(k,\epsilon)$ when $\epsilon$ is small (e.g., when $\epsilon<1/7$).

\begin{theorem}\label{main}
If $n \geq
\left(16/\epsilon\right)^{2k+1}$, then every $2$-edge-coloring of $K_n$
with at least $\epsilon{n \choose 2}$ edges in each color contains a member of $\mathcal{F}_k$.
\end{theorem}

\noindent Using a simple probabilistic argument, one can show that the bound in this theorem is tight up to a constant factor
in the exponent for all $k$ and $\epsilon$ (see Proposition \ref{probprop}).

In this paper, we also consider an analogous problem for tournaments.
A {\it tournament} is a directed graph obtained by choosing a
direction for each edge in an undirected complete graph. A
tournament is {\it transitive} if there is a linear ordering of the
vertices such that all edges point in one direction. Erd\H{o}s and
Moser \cite{ErMo} obtained a variant of Ramsey's theorem for tournaments. Let
$T(k)$ denote the minimum $n$ such that every tournament on $n$
vertices contains a transitive subtournament of order $k$. They proved that $2^{(k-1)/2} \leq T(k) \leq 2^{k-1}$.

In order to prove an analogue of Bollobas's conjecture for tournaments, the following definition appears to be the most natural.
A tournament on $n$ vertices is {\it $\epsilon$-far from being transitive} if the direction of
at least $\epsilon n^2$ edges need to be changed to obtain a transitive tournament. We think this is a good analogue
of $2$-edge-colorings in which both colors are sufficiently represented. Let $D_k$ denote the tournament with
$3k$ vertices formed by taking three disjoint transitive tournaments each of size $k$ with vertex sets $V_0,V_1,V_2$, and
 directing all edges in $V_i \times V_{i+1}$ from $V_i$ to $V_{i+1}$ for $i=0,1,2$, where indices are taken modulo $3$.
By considering the tournament $D_{n/3}$ with $n$ a multiple of $3$, which is $\frac{1}{9}$-far from being transitive,
we see that $D_k$ is essentially the only possible type of unavoidable tournament in tournaments
that are far from being transitive. We prove the following variant of Bollobas's conjecture for tournaments.

\begin{theorem}\label{tournamentthm}
There is a constant $c$ such that if $n \geq \epsilon^{-ck/\epsilon^2}$ and $T$ is a tournament on $n$ vertices
which is $\epsilon$-far from being transitive, then $T$ contains $D_k$.
\end{theorem}

It is easy to see that if a tournament on $n$ vertices contains $\epsilon n^3$
directed triangles, then it is $\epsilon/3$-far from being transitive. Indeed, each edge
is in at most $n$ directed triangles, so a tournament with $\epsilon n^3$ directed
triangles contains $\frac{\epsilon}{3}n^2$ edge-disjoint directed triangles. The
direction of at least one edge of each such directed triangle needs to be changed in
order to make the tournament transitive. In the other direction, the following lemma demonstrates that if a tournament is
$\epsilon$-far from being transitive, then it contains many directed triangles. We will use this lemma in the proof
of Theorem \ref{tournamentthm}.

\begin{lemma}\label{farmany}
There is a constant $c>0$ such that if a tournament $T$ on $n$ vertices is $\epsilon$-far from being transitive, then it
contains at least $c\epsilon^2n^3$ directed triangles.
\end{lemma}

The following simple tournament construction demonstrates that the bound in Lemma \ref{farmany} is best possible up
to the constant factor $c$. Let $V_1,\ldots,V_d$ be disjoint vertex sets of size $3n$, and direct all edges from
$V_i$ to $V_j$ for $i<j$. The tournament restricted to each $V_i$ is a copy of $D_n$. Note that each $V_i$ contains $n^2$
edge-disjoint directed triangles and $n^3$ total directed triangles.
The tournament on $V_1 \cup \ldots \cup V_d$ we constructed has $N=3dn$ vertices,
is $\epsilon$-far from being transitive with $\epsilon=\frac{dn^2}{N^2}=\frac{1}{9d}$, and has only
$dn^3=3\epsilon^2 N^3$ directed triangles.

The statement of Lemma \ref{farmany} appears similar to the well
known triangle removal lemma of Ruzsa and Szemer\'edi \cite{RuSz}.
This lemma states that for every $\epsilon>0$ there is a constant
$\delta>0$ depending only on $\epsilon$ such that if a graph on $n$
vertices is $\epsilon$-far from being triangle-free, i.e., at least
$\epsilon n^2$ edges need to be deleted to make the graph
triangle-free, then it contains at least $\delta n^3$ triangles.
However, as $\epsilon$ tends to zero, in the triangle removal lemma
$\delta$ goes to $0$ much faster than in Lemma \ref{farmany}.

The rest of the paper is organized as follows. In the next section we give a simple proof of Bollobas's conjecture
and show that our upper bound on $n(k,\epsilon)$ is tight up to a constant factor in the exponent for all $k$ and $\epsilon$.
In Section \ref{section3} we study the tournaments that are far from being transitive and show that sufficiently large
such tournaments contain a copy of $D_k$. The last
section of this paper contains some concluding remarks, including a discussion of the possible generalization
of our theorems to hypergraphs. Throughout the paper, we systematically omit floor and ceiling signs
whenever they are not crucial for the sake of clarity of
presentation. We also do not make any serious attempt to optimize
absolute constants in our statements and proofs.

The {\it edge density} of a graph $G$ is the fraction of all pairs of vertices of $G$ that are edges.
In a red-blue edge coloring of $K_n$, the {\it
red neighborhood} $N_R(v)$ of a vertex $v$ is the set of vertices
adjacent to $v$ in color red, and the {\it red degree} of $v$ is $|N_R(v)|$. For a vertex subset $U$ of a graph
$G$, the {\it red common neighborhood} $N_R(U)$ is the set of
vertices adjacent in color red to all vertices in $U$. We similarly
define $N_B(v)$, {\it blue degree}, and $N_B(U)$.

\section{Proof of Theorem \ref{main}}

We first need a lemma which shows that any $2$-edge-coloring with edge density at least
$\epsilon$ in each color has many vertices with large degree in
both colors.

\begin{lemma}\label{start}
In every red-blue edge coloring of $K_n$ ($n \geq 4$) in which each
color has density at least $\epsilon$, the subset $S \subset
V(K_n)$ of vertices that have degree at least $\frac{\epsilon}{4} n$
in both colors has cardinality at least $\frac{\epsilon}{2} n$.
\end{lemma}
\begin{proof}
Assume for contradiction that $|S| <\frac{\epsilon}{2} n$. Note that $\epsilon \leq 1/2$ since each color has density at
least $\epsilon$. Let $R
\subset V(K_n)$ be those vertices with blue degree less than $\frac{\epsilon}{4}n$,
and $B \subset V(K_n)$ be those vertices with red degree less than $\frac{\epsilon}{4}n$. Then $V(K_n)=S
\cup R \cup B$ is a partition of the vertex set. Assume without loss
of generality that $|R| \geq |B|$, so $|R| \geq \frac{1}{2}(n-|S|) > \frac{3}{8}n$. The number of edges between $R$
and $B$ of color red is less than $\frac{\epsilon}{4} n |B|$ since
each vertex in $B$ has red degree less than $\frac{\epsilon}{4} n$.
However, the number of red edges between $R$ and $B$ is
greater than $\left(|B|-\frac{\epsilon}{4}n\right)|R|$ since every vertex in
$R$ has blue degree less than $\frac{\epsilon}{4}n$. Hence,
$$\frac{\epsilon}{4} n|B|>\left(|B|-\frac{\epsilon}{4} n\right)|R|>\left(|B|-\frac{\epsilon}{4}n\right)\frac{3}{8}n.$$
Rearranging the terms and using $\epsilon \leq 1/2$, we have
$$\frac{\epsilon}{4}n \cdot \frac{3}{8}n> |B|\left(\frac{3}{8}n-\frac{\epsilon}{4}n\right)\geq |B|\frac{n}{4}.$$
Multiplying both sides by $4/n$, we have $|B| < \frac{3\epsilon}{8}n$.
We now give an upper bound on the number of edges of color blue. Each vertex in
$R$ has degree at most $\frac{\epsilon}{4} n$ in color blue, so
these vertices participate in at most $\frac{\epsilon}{4}n|R|$ blue
edges. The number of remaining edges, those contained in $B \cup S$,
is ${|B|+|S| \choose 2}$. Let $x=|B|+|S|$, so $x < \epsilon n$. Hence, the number of edges of color blue
is at most
$$\frac{\epsilon}{4}n|R|+{|B|+|S| \choose
2}=\frac{\epsilon}{4}n(n-x)+{x \choose
2}<\frac{\epsilon}{4}n^2+\frac{\epsilon^2}{4}n^2
\leq \epsilon {n \choose 2},$$ contradicting that the blue edge density
is at least $\epsilon$ and completing the proof.
\end{proof}

The following lemma is a variant of the results in \cite{KoRo} and
\cite{Su}. Its proof uses a probabilistic argument commonly referred
to as dependent random choice, which appears to be a powerful tool
in proving various results in Ramsey theory (see, e.g., \cite{FoSu}
and its references).

\begin{lemma}\label{drc}
For a red-blue edge coloring of $K_n$, let $S$ denote the vertex
subset $S \subset V(K_n)$ such that every vertex in $S$ has degree
at least $\alpha n $ in each color, and $s=|S|$. If $\beta\leq
s^{-k/h}$, then there is a subset $T \subset S$ with $|T| \geq
\frac{1}{2}\alpha^h(1-\alpha)^h s - 2$ such that every $k$-tuple in
$T$ has at least $\beta n$ common neighbors in each color.
\end{lemma}
\begin{proof}
We may assume $\alpha \leq 1/2$ and $n \geq 4^{h+1}$, since otherwise we can let $T$ be the empty set.
Let $U_1=\{x_1,\ldots,x_h\}$, $U_2=\{y_1,\ldots,y_h\}$ be two
subsets of $h$ vertices from $V(K_n)$ chosen uniformly at random with
repetitions, and let $W=N_R(U_1) \cap N_B(U_2) \cap S$. That is, $W$
is the set of vertices of $S$ that are adjacent in red to every
vertex in $U_1$ and adjacent in blue to every vertex in $U_2$. Since each vertex $v \in S$ has degree at least
$\alpha n$ in each color and $\alpha \leq 1/2$, then
$$|N_R(v)||N_B(v)| \geq \alpha n \left((1-\alpha)n-1\right) \geq \alpha(1-\alpha)(1-2/n)n^2$$
for each $v \in S$.
Therefore, using $n \leq 4^{h+1}$, we have
\begin{eqnarray*}\mathbb{E}[|W|] & = & \sum_{v \in S}{\bf Pr}(v \in N_R(U_1)){\bf Pr}(v \in N_B(U_2))=\sum_{v \in
S}\left(\frac{|N_R(v)|}{n}\right)^h\left(\frac{|N_B(v)|}{n}\right)^h
\\ & \geq & \alpha^{h}(1-\alpha)^{h}(1-2/n)^hs \geq \frac{1}{2}\alpha^{h}(1-\alpha)^{h}s.\end{eqnarray*}

The probability that a given set $D\subset S$ of vertices is adjacent to $U_1$ by red edges is
$\left(\frac{|N_R(D)|}{n}\right)^h$. Let $Z$ denote the number of
$k$-tuples in $W$ with less than $\beta n$ common red neighbors. So
$$\mathbb{E}[Z]=\sum_{D \subset S, |D|=k,|N(D)|<\beta n} {\bf Pr}(D
\subset N_R(U_1)) \leq {s \choose k}\beta^h \leq s^k \beta^h \leq
1.$$ We similarly have that the expected number, which we denote by
$Z'$, of $k$-tuples in $W$ with less than $\beta n$ common blue
neighbors is at most $1$.

Hence, the expectation of $|W|-Z-Z'$ is at least
$\frac{1}{2}\alpha^h(1-\alpha)^h s-2$ and thus there are choices for
$U_1$ and $U_2$ such that the corresponding value of $|W|-Z-Z'$ is
at least $\frac{1}{2}\alpha^h(1-\alpha)^h s - 2$. Delete a vertex
from every $k$-tuple $D \subset W$ with less than $\beta n$ red
common neighbors or with less than $\beta n$ blue common neighbors.
Letting $T$ be the resulting set, it is clear that $T$ has the
desired properties, completing the proof.
\end{proof}

Having finished all the necessary preparations, we are now ready to prove our first theorem.
Recall that the Ramsey number $R(k)$ is the least positive integer $n$ such that every red-blue edge coloring of the complete graph
$K_n$ contains a monochromatic $K_k$.  We will use the simple bound $R(k) < 4^k$ which was mentioned in the
introduction.

\vspace{0.2cm}
\noindent {\bf Proof of Theorem \ref{main}:} Consider a red-blue edge coloring
of $K_n$ with $n \geq (16/\epsilon)^{2k+1}$ in which both colors
have density at least $\epsilon$. By Lemma \ref{start}, there is a
vertex subset $S$ with $|S|\geq \frac{\epsilon}{2}n$ in which
every vertex has degree at least $\frac{\epsilon}{4}n$ in each
color. Apply Lemma \ref{drc} with $\alpha=\frac{\epsilon}{4}$,
$\beta=\frac{R(k)}{n}$, $h=2k$, and $s=|S|\geq \frac{\epsilon}{2}n$.
We can apply this lemma since
$$\beta=\frac{R(k)}{n}<4^k/n< n^{-1/2}\leq s^{-k/h}.
$$
 So there is a subset $T \subset S$ with
 $$|T| \geq \frac{1}{2}\left(\alpha(1-\alpha)\right)^{2k}s-2 \geq \left(\frac{\epsilon}{8}\right)^{2k+1}n \geq 4^k \geq
 R(k)$$
and every $k$-tuple $D \subset T$ has at least $\beta n=R(k)$ common
neighbors in each color. Since $|T| \geq R(k)$, there is a
monochromatic $k$-clique $D$ in $T$. Assume without loss of
generality that the color of this monochromatic $k$-clique is red.
By our construction of the set $T$, $|N_B(D)| \geq \beta n=R(k)$, so there is a monochromatic
$k$-clique $X$ in $N_B(D)$. This gives us a coloring from ${\cal F}_k$, since
the $k$-clique $D$ is red,
the edges between $D$ and $X$ are blue, and $X$ is also monochromatic, completing the proof.
 \qed

Next we show that the bound in Theorem \ref{main} is tight up to the constant factor in the exponent.

\begin{proposition}\label{probprop}
If $\epsilon \leq 1/2$ and $k \geq 2$, then $n(k,\epsilon)> \epsilon^{-(k-1)/2}$.
\end{proposition}
\begin{proof}
Consider a red-blue edge coloring of $K_n$ with $n =\epsilon^{-(k-1)/2}$, where the red graph is
taken uniformly at random from all labeled graphs with $n$ vertices and $m=\epsilon{n \choose 2}$ edges.
Since
$\epsilon \leq 1/2$, then both the red graph and the blue graph each have edge density at least
$\epsilon$. Note also, that for any collection of $\ell$ edges of $K_n$ the probability that the red
graph contains all these edges is
$$\frac{{{n \choose 2}-\ell \choose m -\ell}} {{{n \choose 2} \choose m}}=
\prod_{i=0}^{\ell-1}\frac{m-i}{{n \choose 2}-i}\leq \left(\frac{m}{{n \choose 2}}\right)^\ell=
\epsilon^\ell.$$

Therefore any copy of a complete bipartite graph with parts of size $k$ (and with $k^2$ edges),
has probability at most $\epsilon^{k^2}$ of being  red. By linearity of
expectation, the expected number of red copies of the complete bipartite graph with parts of size $k$ is
at most $$\frac{1}{2}{n \choose k}{n-k \choose k}\epsilon^{k^2} \leq \frac{n^{2k}}{2k!^2}\epsilon^{k^2}
\leq \frac{\epsilon^k}{2k!^2} \leq \frac{1}{32}.$$ Similarly, the expected number of red cliques with
$k$ vertices is at most $${n \choose k}\epsilon^{{k \choose 2}} < \frac{n^k}{k!}\epsilon^{{k \choose 2}}
= \frac{1}{k!} \leq \frac{1}{2}.$$

Therefore, there is a red-blue edge coloring of $K_n$ with each color having edge density at least $\epsilon$, with
no monochromatic red clique of order $k$, and with no monochromatic red complete bipartite graph with parts of size $k$.
In particular, this edge coloring of $K_n$ does not contain a member of $\mathcal{F}_k$, which implies that
$n(k,\epsilon)>n=\epsilon^{-(k-1)/2}$.
\end{proof}

\section{Proof of Theorem \ref{tournamentthm}} \label{section3}

We start this section with a lemma that shows that every tournament has many vertices
with large outdegree and large indegree.

\begin{lemma}\label{firsttournamentlemma}
Every tournament $T$ on $n \geq 2$ vertices has more than $\frac{n}{3}-2$ vertices with outdegree and indegree at least $n/6$.
\end{lemma}
\begin{proof}
Let $A$ be the set of vertices of $T$ with outdegree less than $n/6$ (and hence indegree more than $5n/6-1$),
$B$ be the set of vertices of $T$ with indegree less than $n/6$ (and hence outdegree more than $5n/6-1$), and $C$ be the
set of vertices of $T$ with outdegree and indegree at least $n/6$.
The sets $A,B,C$ partition the vertex set of $T$, so $|A|+|B|+|C|=n$. Assume for contradiction that
$|C| \leq \frac{n}{3}-2$, so $|A|+|B|\geq\frac{2n}{3}+2$.
We may assume without loss of generality that $|A| \geq |B|$, so $|A| \geq \frac{n}{3}+1$. By averaging, the tournament
on $A$ contains a vertex with outdegree at least $\frac{|A|-1}{2} \geq \frac{n}{6}$. This contradicts the definition of $A$, and completes
the proof.
\end{proof}

Let $F(n,t)$ denote the minimum $m$ such that every tournament with $n$ vertices and
at most $t$ directed triangles can be made transitive by changing the direction of at most $m$ edges. We prove the following
recursive bound on $F(n,t)$.

\begin{lemma}\label{astep}
For $n \geq 12$ and $t$, there are $n_1,n_2$ with $n_1+n_2 =n-1$ and $n_1,n_2 \leq \frac{5n}{6}-1$ and $t_1,t_2$ with
$t_1+t_2 \leq t$ such that \begin{equation} \label{useeq}
F(n,t) \leq \frac{18t}{n}+F(n_1,t_1)+F(n_2,t_2).\end{equation}
\end{lemma}
\begin{proof}
Let $T$ be a tournament on $n$ vertices with at most $t$ directed triangles. By Lemma \ref{firsttournamentlemma},
there is a vertex subset $C$ of tournament $T$ of size at least $\frac{n}{3}-2 \geq \frac{n}{6}$ such that every vertex in
$C$ has indegree and outdegree at least $n/6$. By averaging, there is a vertex $v \in C$ that is in at most
$3t/|C| \leq 18t/n$ directed triangles. Let $I$ be the set of inneighbors of $v$
and $O$ be the set of outneighbors of $v$. Then $|I|+|O|=n-1$, $|I|,|O| \geq n/6$, and
 the number of edges from $O$ to $I$ is the number of directed triangles containing $v$, which is at most $18t/n$.

By reversing the direction of all edges directed from $O$ to $I$, and letting $n_1=|I|$, $n_2=|O|$, $t_1$ denote the number of directed triangles whose vertices are all contained in $I$,
and $t_2$ denote the number of directed triangles whose vertices are all contained in $O$, we have established (\ref{useeq}).
\end{proof}

Note that every tournament on $n$ vertices can be made transitive by changing the direction of at most half of its edges.
So for all $n$ and $t$, we have
\begin{equation}\label{secondeqtour}
F(n,t) \leq {n \choose 2}/2 < n^2/4.
\end{equation}

The following statement implies Lemma \ref{farmany}.

\begin{lemma}
For all $n$ and $t$ we have $F(n,t) \leq 27(nt)^{1/2}$.
\end{lemma}
\begin{proof}
Let $T$ be a tournament with $n$ vertices and at most $t$ directed triangles. For a vertex subset $W \subset V(T)$, let
$t(W)$ denote the number of directed triangles contained in $W$.

Consider the following recursive procedure. At step $i$, we will have a partition $\mathcal{F}_i$ of $V(T)$
into disjoint sets, a tournament $T_i$, and a linear order on these sets such that all edges between different subsets in $\mathcal{F}_i$
are directed forward in $T_i$. Let $T_0=T$ and $\mathcal{F}_0=\{V(T)\}$. For all $W \in \mathcal{F}_i$,
 if $|W| \leq 36(t/n)^{1/2}$, then put $W \in \mathcal{F}_{i+1}$. Otherwise,
 we find a partition $W=W_1 \cup W_2$ of $W$ into two subsets each of cardinality at least $|W|/6 > 6(t/n)^{1/2}$
 according to Lemma \ref{astep} and change the direction of at most $18t(W)/|W|$ edges in $T_i$
 to make all edges between $W_1$ and $W_2$ directed from $W_1$ to $W_2$, update the linear ordering by placing $W_1$
 immediately before $W_2$ in the same place $W$ was, and put $W_1,W_2 \in \mathcal{F}_{i+1}$.
It is clear that the resulting partition $\mathcal{F}_{i+1}$ of $V(T)$ and tournament $T_{i+1}$
will also satisfy that there is a linear order on the sets in $\mathcal{F}_{i+1}$
such that all edges between different subsets in $\mathcal{F}_{i+1}$ are directed forward in $T_{i+1}$. Eventually, there is a step $j$
such that each set in the partition $\mathcal{F}_j$ has cardinality at most $36(t/n)^{1/2}$. We let $\mathcal{F}=\mathcal{F}_j$ and
$T^*=T_j$.

Following this procedure, we build a {\em tree} $\cal{T}$ of subsets of $V(T)$,
with $V(T)$ being its root, so that $W \in \mathcal{F}_{i}$ is a leaf of $\cal{T}$ if $|W| \leq 36(t/n)^{1/2}$ and
otherwise $W$ has two children $W_1,W_2$ as described above. Notice that $\mathcal{F}$ consists of the leaves of $\cal{T}$.

The vertices of $\cal{T}$ that contain any particular directed triangle in
$T$ form a path in $\cal{T}$ starting from $V(T)$. For each directed triangle $D$ in $T$, define its weight
$$w(D)=\sum_{D \subset W, W \in V(\cal{T})} 18/|W|.$$
Recall that if $W$ was not a leaf of $\cal{T}$, then we changed the
direction of at most $18t(W)/|W|$ edges between $W_1$ and $W_2$. So
a particular directed triangle $D$ in $W$ contributes at most
$18/|W|$ to the number of edges between $W_1$ and $W_2$ that are
changed. Hence, the sum of the weights of the directed triangles in
$T$ is an upper bound on the number of edges whose direction is
changed when creating $T^*$ from $T$. If $W^0=V(T),\ldots,W^i$ are
the vertices of the path in $\cal{T}$ that contain $D$, then
$$w(D)=18/|W^i|+18/|W^{i-1}|+\cdots+18/|W^0| \leq \frac{18}{6(t/n)^{1/2}}\left(1+5/6+(5/6)^2+\cdots \right)=18(n/t)^{1/2}.$$
So we changed at most $18(n/t)^{1/2}t=18(tn)^{1/2}$ edges to obtain $T^*$ from $T$.

The partition $\mathcal{F}$ and tournament $T^*$ satisfy that each set in $\mathcal{F}$
has cardinality at most $36(t/n)^{1/2}$ and there is a linear order of the sets in $\mathcal{F}$ such that all edges between
different sets in $\mathcal{F}$ are directed forward in $T^*$.
 Using the bound (\ref{secondeqtour}), for each $W \in \mathcal{F}$, we can change the direction of less than
 $|W|^2/4$ edges to make $W$ transitive. Since $|W| \leq 36(t/n)^{1/2}$ for each $W \in \mathcal{F}$, the number of edges
 in sets in $\mathcal{F}$ whose direction in $T^*$ is changed is at most $$\sum_{W \in \mathcal{F}}|W|^2/4 \leq \frac{n}{36(t/n)^{1/2}}(36(t/n)^{1/2})^2/4=9(tn)^{1/2},$$
 where we used that if $\sum x_i=N$ and $0 \leq x_i \leq a$ for each $x_i$, then $\sum x_i^2 \leq (N/a)a^2$.
The resulting tournament is clearly transitive, and we have changed
the direction of at most $(18+9)(tn)^{1/2}=27(tn)^{1/2}$ edges of $T$ to obtain this tournament.
\end{proof}

We will also need the following lemma, which is a variant of Lemma \ref{drc}.

\begin{lemma}\label{drc2} Let $H$ be a bipartite graph with parts $V_1$ and $V_2$ each of size $n$ and
with at least $\alpha n^2$ edges. If $\beta \leq n^{-d/h}$, then there is a subset $W_1 \subset V_1$
such that $|W_1| \geq \alpha^h n-1$ and every $d$-tuple in $W_1$ has at least $\beta n$ common
neighbors. \end{lemma} \begin{proof} Let $A=\{x_1,\ldots,x_h\}$ be a subset of $h$ vertices chosen
uniformly at random with repetitions from $V_2$, and let $U_1=N(A)$. By linearity of expectation and
convexity of $f(z)=z^h$, we have $$\mathbb{E}[|U_1|]=\sum_{v \in V_1}\left(\frac{|N(v)|}{n}\right)^h
=n^{-h}\sum_{v \in V_1}|N(v)|^h \geq n^{1-h} \left(\frac{\sum_{v \in V_1}|N(v)|}{n}\right)^h\geq
\alpha^h n.$$ The probability that a given set $D \subset V_1$ of vertices is contained in $U_1$ is
$\left(\frac{|N(D)|}{n}\right)^h$. Letting $Y$ denote the number of $d$-tuples in $W_1$ whose common
neighborhood has size less than $\beta n$, we have
 $$\mathbb{E}[Y]<{n \choose d}\beta^h \leq n^d\beta^h \leq 1.$$ Hence, the expectation of $|U_1|-Y$ is more than
 $\alpha^h n-1$, and thus there is a choice $A_0$ of $A$ such that the corresponding value of $|U_1|-Y$
 is more than $\alpha^h n-1$.
Delete a vertex from every $d$-tuple $D$ in $U_1$ with less than $\beta n$ common
 neighbors. Letting $W_1$ be the resulting subset of $U_1$, it
 is clear that $W_1$ has the desired properties.
\end{proof}

The classical problem of Zarankiewicz in extremal graph theory asks: given $m,n,s,$ and $t$, what is the maximum number $z(m,n;s,t)$ of edges
of a bipartite graph with first part of size $m$
and second part of size $n$ if no $s$ vertices of the first part have $t$ vertices in common? A simple
double-counting argument (see page 310 of \cite{Bo})
demonstrates that \begin{equation}\label{Zar}
z(m,n,s,t)<(s-1)^{1/t}(n-t+1)m^{1-1/t}+(t-1)m.\end{equation}
We will use this bound in the proof of the following theorem, which, together with Lemma \ref{farmany},
establishes Theorem \ref{tournamentthm}.

\begin{theorem}
If $T$ is a tournament with $n$ vertices and $\delta n^3$ directed triangles with $n \geq \delta^{-4k/\delta}$, then $T$
contains $D_k$.
\end{theorem}
\begin{proof}
Consider a random partition $V(T)=V_0 \cup V_1 \cup V_2$ of the
vertex set of $T$ into three disjoint subsets each of size $n/3$.
Then for every directed triangle $D$ in $T$ the probability that it
satisfies $|V_i \cap D|=1$ for $i=0,1,2$ and the edge between $V_1
\cap D$ and $V_2 \cap D$ is directed from $V_1 \cap D$ to $V_2 \cap
D$ is at least $1/9$. By linearity of expectation, there is a
partition $V(T)=V_0 \cup V_1 \cup V_2$ with $|V_i|=n/3$ for
$i=0,1,2$ for which there are at least $\delta n^3/9$ triples
$(v_0,v_1,v_2) \in V_0 \times V_1 \times V_2$ that form a directed
triangle such that the direction of the edge between $v_1$ and $v_2$
is from $v_1$ to $v_2$. We call such a triple $(v_0,v_1,v_2) \in V_0
\times V_1 \times V_2$ a {\it positively oriented directed
triangle}. Consider the bipartite graph $H$ with parts $V_1$ and
$V_2$, where $(v_1,v_2) \in V_1 \times V_2$ is an edge of $H$ if and
only if there are at least $\delta n/2$ vertices $v_0 \in V_0$ for
which the triple $(v_0,v_1,v_2)$ is a positively oriented directed
triangle. Note that if $(v_1,v_2) \in V_1 \times V_2$ is an edge of
$H$, then the edge between $v_1$ and $v_2$ in $T$ is directed from
$v_1$ to $v_2$. The number of positively oriented directed triangles
which do not contain an edge of $H$ is at most $|V_1||V_2|\delta
n/2=\delta n^3/18$. Hence, the number of positively oriented
directed triangles in $V_0 \times V_1 \times V_2$ containing an edge
of $H$ is at least $\delta n^3/9-\delta n^3/18=\delta n^3/18$, and
the number of edges of $H$ is at least $$\frac{\delta
n^3/18}{|V_0|}= \delta n^2/6.$$ So the fraction of pairs of $V_1
\times V_2$ that are edges of $H$ is at least $\frac{\delta
n^2/6}{|V_1||V_2|}=3\delta/2$.

By Lemma \ref{drc2} with $d=k/\delta$, $h=2d$ and $\beta=n^{-d/h}=n^{-1/2}$, there is a subset $W_1
\subset V_1$ of size at least $(3\delta/2)^{2d}|V_1|-1 \geq n^{1/2}$
such that every $d$-tuple in $W_1$ has at least $n^{1-d/h}=n^{1/2}$ neighbors in $H$. Since $n^{1/2} \geq 2^{d}$,
the result of Erd\H{o}s and Moser (mentioned in the introduction) implies that there is a transitive subtournament in $W_1$ of size $d$. Let $X_1=\{x_1,\ldots,x_d\}$ denote
the set of vertices of this transitive subtournament. The common neighborhood $N_H(X_1)$ of $X_1$ in graph $H$ has size at least
$n^{1/2}$, so there is a transitive subtournament in $N_H(X_1)$ of size $d$. We let $X_2=\{y_1,\ldots,y_d\}$
denote this transitive subtournament. Since $X_1$ and $X_2$ form a complete bipartite graph in $H$, then every
edge between $X_1$ and $X_2$ in tournament $T$ directs from $X_1$ to $X_2$.

To complete the proof, we will use a variant of an argument from \cite{Ni}.
Consider the bipartite graph $F$ with parts $V_0$ and $\{1,\ldots,d\}$, where $v_0 \in V_0$ is
adjacent to $i$ if and only if $(v_0,x_i,y_i)$ is a positively oriented directed triangle. Since each pair
$(x_i,y_i)$ is an edge of $H$, each $i \in \{1,\ldots,d\}$ is adjacent to at least $\delta n/2$ vertices
in $V_0$ in $F$. So the number of edges of $F$ is at least $d\delta n/2 =kn/2$. By (\ref{Zar}), $F$ contains
a complete bipartite graph with $n^{1/2}$ vertices in $V_0$ and $k$ vertices in $\{1,\ldots,d\}$,
since otherwise the number of edges of $F$ would be at most
\begin{eqnarray*}
z(n/3,d,n^{1/2},k) & < & (n^{1/2}-1)^{1/k}(d-k+1)(n/3)^{1-1/k}+(k-1)n/3<dn^{1-\frac{1}{2k}}+kn/3 \\ & = &
(k/\delta)n^{1-\frac{1}{2k}}+kn/3 \leq
\frac{\delta^{2/\delta}}{\delta}kn+kn/3<kn/2,\end{eqnarray*}
contradicting that $F$ has at least $kn/2$ edges (here recall, that by definition, $\delta \leq 1/6$).

So there is a subset $X_0  \in V_0$ of size at least $n^{1/2}$ and a subset $R=\{i_1,\ldots,i_k\} \subset \{1,\ldots,d\}$
such that every vertex in $X_0$ is adjacent to every vertex in $R$ in graph $F$. By construction,
every triple $(x,x_{i_j},y_{i_{j}})$
with $x \in X_0$ and $1 \leq j \leq k$ forms a positively oriented directed triangle.
Therefore, for each $x \in X_0$ and $1 \leq j \leq k$, the edge between $x$ and $x_{i_j}$ is directed from  $x$ to
$x_{i_j}$ and the edge between $y_{i_j}$ and $x$ is directed from $y_{i_j}$ to $x$. Since $n^{1/2} \geq 2^{d}$,
the result of Erd\H{o}s and Moser implies that there is a transitive subtournament in $X_0$ of size $k$. Let $U_0$
be the set of vertices of this transitive subtournament, $U_1=\{x_{i_1},\ldots,x_{i_k}\}$, and $U_2=\{y_{i_1},\ldots,y_{i_k}\}$.
Since $U_i \subset X_i$ and $X_i$ is a transitive subtournament for $i=1,2$, then $U_i$ is a transitive
subtournament. Since $U_0 \subset X_0$, then every edge between $U_0$ and $U_1$ is directed from $U_0$
to $U_1$ and every edge from $U_2$ to $U_0$ is directed from $U_2$ to $U_0$. Recall that every edge
between $X_1$ and $X_2$ has direction from $X_1$ to $X_2$, so every edge between $U_1$ and $U_2$
is directed from $U_1$ to $U_2$. We have shown that the three sets $U_0,U_1,U_2$ form a copy of $D_k$
in $T$. This completes the
proof. \end{proof}

\section{Concluding Remarks}

Although using a simple probabilistic argument it is easy to show that
the bound in Theorem \ref{tournamentthm} has the correct exponential dependence on $k$, we think the dependence
on $\epsilon$ can be further improved. Let $T(k,\epsilon)$ be the minimum positive integer
such that every tournament $T$ with $n \geq T(k,\epsilon)$ vertices that is $\epsilon$-far from being transitive contains
the tournament $D_k$.

\begin{conjecture}
There is a constant $c$ such that $T(k,\epsilon) \leq \epsilon^{-ck}$.
\end{conjecture}

It is not difficult to extend our proof of Theorem \ref{main} to establish the following
hypergraph generalization of Bollobas's conjecture.

\begin{theorem}
For each $\epsilon>0$ and positive integers $r$ and $k$, there is a minimum positive integer
$N=N(k,r,\epsilon)$ such that every $2$-edge-coloring of the complete $r$-uniform hypergraph with $n
\geq N$ vertices and at least $\epsilon{n \choose r}$ edges in each color contains disjoint vertex
subsets $V_1,\ldots,V_r$ each of size $k$ such that for every function $f:\{1,\ldots,r\} \rightarrow
\{1,\ldots,r\}$, all the edges $(v_1,\ldots,v_r)$ with $v_i \in V_{f(i)}$ for $1 \leq i \leq r$ have the
same color, but the edge coloring of $V_1 \cup \ldots \cup V_r$ is not monochromatic. \end{theorem}

The {\it tower function} $t_i(x)$ is defined by $t_1(x)=x$ and $t_i(x)=2^{t_{i-1}(x)}$ for $i \geq 2$. The
{\it hypergraph Ramsey number} $R_r(k)$ is the minimum $N$
such that every $2$-edge-coloring of the complete $r$-uniform hypergraph on $N$ vertices contains a monochromatic
complete $r$-uniform hypergraph on $k$ vertices. It is known (see \cite{GrRoSp}) that $R_r(k) \leq t_{r}(ck)$, where the constant $c$ depends on $r$.
It would be interesting to prove a similar upper bound for $N(k,r,\epsilon)$.

\vspace{0.2cm} \noindent {\bf Acknowledgment.}\, We would like to thank Jonathan
Cutler for sharing a copy of his paper with Bal\'azs Mont\'agh with
us.

\end{document}